\documentclass[a4paper,10pt]{amsart}

\usepackage[english]{babel}
\usepackage[utf8]{inputenc}

\usepackage{tikz}
\usepackage{mathrsfs}
\usepackage{mathtools}
\usepackage{dsfont}
\usepackage{amssymb}
\usepackage{amsmath}
\usepackage{enumerate}
\usepackage{amsfonts}
\usepackage{amsthm,amsmath}
\usepackage{color}
\usepackage{hyperref}

\usetikzlibrary{arrows.meta}

\numberwithin{equation}{section}
\numberwithin{figure}{section}

\newcommand{\N}{\mathds{N}}

\newcommand{\R}{\mathds{R}}

\newcommand{\cU}{\mathscr{U}}

\newcommand{\applied}[2]{\langle #1,#2\rangle}
\DeclarePairedDelimiter\norm{\lVert}{\rVert}
\DeclarePairedDelimiter\abs{\lvert}{\rvert}
\newcommand{\argument}{\,\cdot\,}
\renewcommand{\phi}{\varphi}

\newcommand{\eps}{\varepsilon}

\DeclareMathOperator{\id}{id}

\definecolor{mygreen}{rgb}{0.1,0.75,0.2}

\numberwithin{equation}{section}

\newtheorem*{thm*}{Theorem}
\newtheorem{thm}{Theorem}[section]

\newtheorem{lem}[thm]{Lemma}

\theoremstyle{remark}
\newtheorem{rem}[thm]{Remark}

\newtheorem{example}[thm]{Example}

\theoremstyle{definition}
\newtheorem{defn}[thm]{Definition}

\begin{document}

\title{Mean ergodicity vs weak almost periodicity}
\author{Moritz Gerlach}
\address{Moritz Gerlach\\Universit\"at Potsdam\\Institut f\"ur Mathematik\\Karl-Liebknecht-Stra{\ss}e 24--25\\14476 Potsdam\\Germany}
\email{moritz.gerlach@uni-potsdam.de}

\author{Jochen Gl\"uck}
\address{Jochen Gl\"uck\\Universit\"at Ulm\\Institut f\"ur Angewandte Analysis\\89069 Ulm, Germany}
\email{jochen.glueck@uni-ulm.de}


\date{\today}

\begin{abstract}
We provide explicit examples of positive and power-bounded operators on $c_0$ and $\ell^\infty$ which are mean ergodic but not
weakly almost periodic. As a consequence we prove that a countably order complete Banach lattice on which every positive and power-bounded mean ergodic operator
is weakly almost periodic is necessarily a KB-space.
This answers several open questions from the literature. 
Finally, we prove that if $T$ is a positive mean ergodic operator with zero fixed space on an arbitrary Banach lattice, then so is every power of $T$.
\end{abstract}
\maketitle

\section{Introduction and Preliminaries}

A bounded linear operator $T$ on a Banach space $X$ is called \emph{mean ergodic} if the limit of
its Ces\'aro averages $\lim_{n\to\infty}\frac{1}{n} \sum_{k=0}^{n-1} T^k x$ exists for every $x\in X$.
By the well-known mean ergodic theorem, see e.g.\ \cite[Thm 8.20]{haase2015},
a sufficient condition for a power-bounded operator (and all of its powers) to be mean ergodic is
that the operator is \emph{weakly almost periodic}, meaning that $\{T^n x : n\in\N\}$ is relatively weakly compact for
every $x\in X$. In particular, every power-bounded operator on a reflexive Banach space is mean ergodic. In general, weak almost periodicity is not necessary for mean ergodicity, though; see e.g.\ \cite[Exa 2]{bermudez2000} for a mean
ergodic operator on $C[-1,0]$ whose square is not mean ergodic and who can therefore not be weakly almost periodic, and see \cite[Exa~8.27]{haase2015} for a similar construction on $c_0$.

The problem becomes more subtle if one considers only positive operators on Banach lattices. Sine gave the first example of a positive and contractive 
operator on a $C(K)$-space that is mean ergodic but not weakly almost periodic \cite{sine1976}; see also \cite{emelyanov2009} for another example. 
However, in case that $T$ is a positive, contractive and mean ergodic operator on an $L^1$-space, it was shown independently 
by Komornik \cite[Prop 1.4(i)]{komornik1993} and Kornfeld and Lin \cite[Thm 1.2]{kornfeld2000} that $T$ is weakly almost periodic;
see also \cite[Thm 3.1.11]{emelyanov2007} for an extension of this result to power-bounded operators.

In general, the question on what Banach lattices every positive, power-bounded and mean ergodic operator 
is automatically weakly almost periodic is still open, see \cite[page 131]{emelyanov2007}. In the present article we provide some partial answers to this problem.

In Section \ref{sec:c0}, we give the first counterexample to this question on $c_0$, a Banach lattice with order continuous norm. 
This answers Open Problem 3.1.18 in \cite{emelyanov2007}. In Section \ref{sec:linfty}, we
then construct a counterexample on $\ell^\infty$ which, in addition, answers 
Question~6.(iii) of \cite{emelyanov2009}.
From this we conclude in Theorem \ref{thm:KB} that
every countably order complete Banach lattice on which mean ergodicity and weak almost periodicity are equivalent for positive and power-bounded operators
is necessarily a KB-space.
We also refer to \cite{alpay2006} for a characterization of KB-spaces by mean ergodicity of certain positive operators. 
In the final Section~\ref{sec:ergodic-powers} we prove that all powers of a mean ergodic positive operator $T$ are also mean ergodic in case that the mean ergodic projection of $T$ equals $0$.

Throughout the article, we use the notion \emph{operator} synonymously with \emph{linear operator}.

\section{A counterexample on $c_0$}
\label{sec:c0}

In the following we give an example of a positive, power-bounded and mean ergodic operator on $c_0$ which is not weakly almost periodic. 
Recall that $c_0$, the space of real null sequences over $\N$ endowed with the sup norm, is a Banach lattice with order continuous norm.

We start by noting an easy observation about the norm of positive operators on $c_0$.

\begin{lem}
	\label{lem:basics}
	Every positive operator $T\colon c_0 \to c_0$ extends uniquely to a positive and order continuous operator $S$ on $\ell^\infty$. Moreover, we have
	\[ \norm{T} = \norm{S} = \norm{S\mathds{1}} = \norm{\sup \{ T\mathds{1}_{\{1,\dots,N\}} : N\in \N\}} = \lim_{N\to\infty} \norm{T\mathds{1}_{\{1,\dots,N\}}}.\]
\end{lem}
\begin{proof}
	The uniqueness is clear since $c_0$ is order dense in $\ell^\infty$. To show the existence of $S$, define $S \coloneqq T^{**}$. 
	Obviously, $S$ is a positive extension of $T$ with $\norm{S}=\norm{T}$. As $S$ is weak$^*$-continuous, it easily follows that $S$ is order continuous. 
	Moreover, $\norm{S} = \norm{S\mathds{1}}$ since $\ell^\infty$ is an AM-space with order unit $\mathds{1}$.
	Now, the equation $\norm{S\mathds{1}} = \norm{\sup \{ T\mathds{1}_{\{1,\dots,N\}} : N\in \N\}}$ follows from the order continuity of $S$ and from
	the fact that $\ell^\infty$, like every AM-space with order unit, has the Fatou property \cite[p.\ 65]{abramovich2002},
	meaning that $\norm{f} = \lim_{j} \norm{f_j}$ for every increasing net $(f_j) \subseteq \ell^\infty_+$ with supremum $f$.
\end{proof}

In order to present the announced example in a most accessible way, we first describe (the action of) positive operators on $c_0$ by a weighted (and directed) graph $(\N,w)$ 
with vertex set $\N$ and weight function $w\colon \N^2 \to [0,\infty)$.
More precisely, let $T$ be a positive linear operator on $c_0$; we associate a graph $(\N,w)$ to $T$ in the following way:
every pair of vertices $(u,v)\in \N^2$ is connected by a directed edge of weight 
\[ w(u,v) \coloneqq \big(T(e_u)\big)_v,\]
where
$e_u=\mathds{1}_{\{u\}}$ denotes the canonical unit vector in $c_0$ supported on $u$.
Vividly speaking, $w(u,v)$ is the amount of mass moved from atom $u$ to $v$ by the operator $T$.
In other words, $\big(w(u,v)\big)_{v,u\in \N}$ is the transition matrix of $T$.
The so defined graph possesses the following properties:
\begin{enumerate}[(a)]
\item $\big(w(u,k)\big)_{k\in\N}$ is a null sequence for each $u\in\N$.
\item $\sup_{u\in\N} \sum_{k=1}^\infty w(k,u) < \infty$.
\end{enumerate}
In fact, condition (a) is nothing but a reformulation of the fact that $Te_u = \big(w(u,k)\big)_{k\in\N} \in c_0$
for every $u\in\N$ and condition (b) follows from
\[ \sum_{k=1}^\infty w(k,u) =\lim_{N\to\infty} \big(T\mathds{1}_{\{1,\dots,N\}}\big)_u  \leq \norm{T}\]  for all $u\in\N$.
Conversely, given a weighted graph $(\N,w)$ such that  $w\colon \N^2 \to [0,\infty)$ satisfies conditions (a) and (b), this graph describes the action of a positive  operator on $c_0$. Let us note this as a lemma.

\begin{lem}
\label{lem:c0graphs}
	Let $(\N,w)$ be a graph with weight function $w\colon \N^2 \to [0,\infty)$ that satisfies conditions (a) and (b) above.
	Then 
	\[ T(x_n) \coloneqq \biggl(\sum_{n=1}^\infty x_n w(n,u) \biggr)_{u\in\N}\]
	defines a positive and order continuous operator on $\ell^\infty$ such that $Tc_0 \subseteq c_0$ and both the operator and its restriction to $c_0$ have norm
	\[ \norm{T} = \sup_{u\in \N} \sum_{k=1}^\infty w(k,u) \eqqcolon C.\]
\end{lem}
\begin{proof}
First of all, for every $(x_n) \in \ell^\infty$ it follows from
\[ \norm{T(x_n)}_\infty =  \sup_{u\in\N} \abs[\bigg]{\sum_{k=1}^\infty x_k w(k,u)} \leq \norm{(x_n)}_\infty \sup_{u\in \N} \sum_{k=1}^\infty w(k,u) \leq \norm{(x_n)} C < \infty\]
that $T(x_n)$ is a bounded sequence. 
Therefore, $T$ is a positive operator on $\ell^\infty$ with $\norm{T} = \norm{T\mathds{1}} =  C$. 
To show that $T$ is order continuous, let $(x^{(j)})_{j \in J} = \big( (x_n^{(j)})_{n \in \N}\big)_{j \in J}$ be a decreasing net in $\ell^\infty_+$ with infimum $0$. 
It suffices to prove that $(Tx^{(j)})_{j \in J}$ converges pointwise to $0$, i.e.\  that
\begin{align*}
	\lim_{j} \sum_{n=1}^\infty x_n^{(j)} w(n,k) = 0
\end{align*}
for each $k \in \N$. To this end fix $k \in \N$ and let $\varepsilon > 0$. 
There exists $M > 0$ such that $\norm{x^{(j)}} \le M$ for all $j \in J$ and, as $\sum_{n=1}^\infty w(n,k) < \infty$, we find $N \in \N$ 
such that $\sum_{n=N+1}^\infty w(n,k) < \eps/M$. Hence, $\sum_{n=N+1}^\infty x_n^{(j)} w(n,k) < \eps$ for all $j \in J$. 
As $(x^{(j)})_{j \in J}$ converges pointwise to $0$, this implies that $\sum_{n=1}^\infty x_n^{(j)} w(n,k) < 2\eps$ for all sufficiently large $j$. Thus, $T$ is order continuous.

To show that $T$ leaves $c_0$ invariant, let $(x_n) \in c_0$ and $\eps>0$. We fix $N\in \N$ such that $\abs{x_n} < \eps$ for all $n\geq N$.
By condition (a) we find $K \in \N$ such that $w(n,k) < \eps/N$ for all $1\leq n < N$ and every $k\geq K$.
Then
\begin{align*} \big(T(x_n)\big)_k &= \sum_{n=1}^\infty x_n w(n,k) = \sum_{n=1}^N x_nw(n,k) + \sum_{n=N+1}^\infty x_nw(n,k) \\
&\leq \norm{(x_n)} \eps + \eps \norm{T} \leq \eps(\norm{(x_n)} + C)
\end{align*}
for all $k\geq K$. This shows that $T$ preserves $c_0$. Since $T$ is an order continuous extension of $T|_{c_0}$ to $\ell^\infty$, 
it follows from Lemma \ref{lem:basics} that $T$ and $T|_{c_0}$ have the same norm.
\end{proof}

In view of Lemma \ref{lem:c0graphs} we give the following definition.
\begin{defn}
	A graph $(\N,w)$ with weight function $w\colon \N^2 \to [0,\infty)$ 
	that satisfies conditions (a) and (b) above is called a \emph{$c_0$-graph}.
	The corresponding positive operator $T$ on $c_0$ given by Lemma \ref{lem:c0graphs} is called its \emph{associated operator}.

Now let $(\N,w)$ be a $c_0$-graph and $u,v\in \N$.
We call a finite sequence $p_{u,v} \coloneqq (u_0,u_1,\dots,u_{n-1},u_n)$ of $n+1$ natural numbers a \emph{path from $u$ to $v$ of length $n$}
if $u_0=u$ and $u_n=v$; we call
\[ w(p_{u,v}) \coloneqq \prod_{k=1}^n w(u_{k-1},u_k)\]
the \emph{weight} of such a path $p_{u,v}$ and write $\abs{p_{u,v}}=n$.
\end{defn}

\begin{lem}
\label{lem:normTn}
	Let $(\N,w)$ be a $c_0$-graph with associated operator $T$ on $c_0$. Then
	\[ \norm{T^n} = \sup_{v\in\N} \sum_{u\in \N} \sum_{\abs{p_{u,v}}=n} w(p_{u,v}).\]
\end{lem}
\begin{proof}
	Fix $N\in\N$. We show inductively that 
	\begin{align}
	\label{eqn:Tnnorm}
	\big(T^n\mathds{1}_{\{1,\dots,N\}}\big)_v = \sum_{u=1}^N \sum_{\abs{p_{u,v}}=n} w(p_{u,v})
	\end{align}
	for all $n\in\N$. 
	Having proved this, the assertion follows from Lemma \ref{lem:basics}.
	For $n=1$ equation \eqref{eqn:Tnnorm} follows immediately from the definition of $T$ since
	\[ \big(T \mathds{1}_{\{1,\dots,N\}}\big)_v = \sum_{u=1}^N w(u,v) = \sum_{u=1}^N \sum_{\abs{p_{u,v}}=1} w(p_{u,v}).\]
	If \eqref{eqn:Tnnorm} holds for $n\in\N$, we obtain that
	\begin{align*}
		\big(T^{n+1} \mathds{1}_{\{1,\dots,N\}}\big)_v &= \biggl(T\biggl( \sum_{u=1}^N \sum_{\abs{p_{u,k}}=n} w(p_{u,k})\biggr)_{k\in\N}\biggr)_v
		= \sum_{k=1}^\infty \biggl( \sum_{u=1}^N \sum_{\abs{p_{u,k}}=n} w(p_{u,k})\biggr) w(k,v) \\
		&= \sum_{u=1}^N \sum_{k=1}^\infty \sum_{\abs{p_{u,k}}=n} w(p_{u,k})w(k,v) = \sum_{u=1}^N \sum_{\abs{p_{u,v}}={n+1}} w(p_{u,v})
	\end{align*}
	which completes the proof of \eqref{eqn:Tnnorm}.
\end{proof}

\begin{rem}
	The idea to represent linear operators on sequence spaces as weighted graphs has quite a long history. It was, for instance, used implicitly by Foguel in \cite{foguel1964} to construct a power bounded operator on a Hilbert space which is not similar to a contraction, and it was also used implicitly in \cite[Section~2]{mueller2007} to construct a counterexample related to the Blum--Hanson property. In \cite[Section~4]{oosterhout2009} this approach was worked out in great detail.
\end{rem}

We now describe a $c_0$-graph whose associated operator is power-bounded, mean ergodic but not weakly almost periodic. It
consists of countably many slightly modified copies of the infinite ladder-shaped graph $G_0$ described in Figure \ref{fig:laddergraph}.

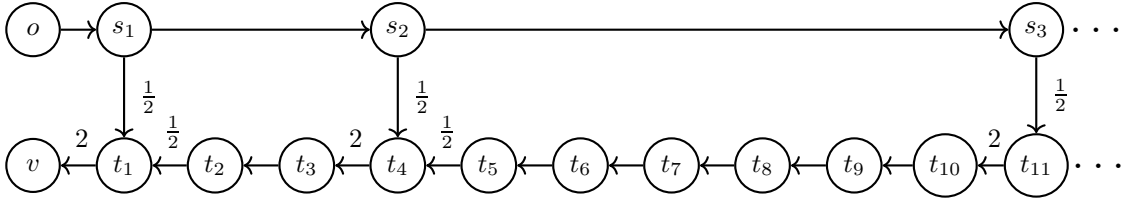
\begin{figure}[h]
\begin{tikzpicture}[scale=0.9] 
\begin{scope}[every node/.style={circle,draw,thick,minimum size=20pt}]
\node (11) at (0,2) {$o$};
\node (12) at (1.33,2) {$s_1$};
\node (13) at (5.33,2) {$s_2$};
\node (14) at (14.66,2) {$s_3$};

\node (21) at (0,0) {$v$};
\node (22) at (1.33,0) {$t_1$};
\node (23) at (2.66,0) {$t_2$};
\node (24) at (4,0) {$t_3$};
\node (25) at (5.33,0) {$t_4$};
\node (26) at (6.66,0) {$t_5$};
\node (27) at (8,0) {$t_6$};
\node (28) at (9.33,0) {$t_7$};
\node (29) at (10.66,0) {$t_8$};
\node (30) at (12,0) {$t_9$};
\node (31) at (13.33,0) {$t_{10}$};
\node (32) at (14.66,0) {$t_{11}$};
\end{scope}

\begin{scope}[every node/.style={circle}]
	\node (15) at (16,2) {};
	\node (33) at (16,0) {};
\end{scope}

\begin{scope}[>={Stealth[black]},every edge/.style={}, every node/.style={font=\huge}]
\path[->] (14) edge node {$\;\hdots$}(15);
\path[->] (32) edge node {$\;\hdots$}(33);
\end{scope}

\begin{scope}[every node/.style={circle}, every edge/.style={draw,thick}]
\path[->] (11) edge (12);
\path[->] (12) edge (13);
\path[->] (13) edge (14);
\path[->] (22) edge node[above]{$\;2$} (21);
\path[->] (23) edge node[above]{$\;\frac{1}{2}$} (22);
\path[->] (24) edge (23);
\path[->] (25) edge node[above]{$\;2$} (24);
\path[->] (26) edge node[above]{$\;\frac{1}{2}$} (25);
\path[->] (27) edge (26);
\path[->] (28) edge (27);
\path[->] (29) edge (28);
\path[->] (30) edge (29);
\path[->] (31) edge (30);
\path[->] (32) edge node[above]{$\;2$} (31);

\path[->] (12) edge node[right]{$\!\!\frac{1}{2}$} (22);
\path[->] (13) edge node[right]{$\!\!\frac{1}{2}$} (25);
\path[->] (14) edge node[right]{$\!\!\frac{1}{2}$} (32);
\end{scope}
\end{tikzpicture}
\caption{The ladder-shaped graph $G_0$.
Every unlabeled edge is weighted with $1$ and every missing edge is to be understood as weighted with $0$.}
\label{fig:laddergraph}
\end{figure}

Obviously, the graph $G_0$ satisfies properties (a) and (b) above and therefore is a $c_0$-graph.
However, the main property of this graph is that the length of paths with non-zero weight from vertex $o$ to $v$ grow exponentially:
if we sort these paths in ascending order by length, then the $n$th path is of length $2^{n+1}-1$. A moment of reflection shows that
for every $n\in\N$ and each vertex $u$ there are at most two paths with non-zero weight of length $n$ with endpoint $u$; the weight of each of them is bounded by $2$.
Thus, by Lemma \ref{lem:normTn}, the associated operator $T_0$ is power-bounded with $\norm{T_0^n}\leq 4$ for all $n\in\N$.

Now we consider countably many modified copies $G_k$, $k\in\N$, of $G_0$ that one obtains by removing the first
finitely many ladder rungs. More precisely, the $k$th copy $G_k$ is lacking the vertices $s_1,\dots,s_k$ and all their corresponding edges;
instead we add an uninterrupted edge from $o$ to $s_{k+1}$ of weight $1$, cf.\ Figure \ref{fig:laddergraphG_2}.
If $T_k$ denotes the operator associated with the graph $G_k$ for any $k\in\N_0$,  then one easily observes
that $(T^n_k e_o)_v = 1$ if $n=2^{m+2}-k-1$ for some $m\geq k$ and $0$ otherwise.

\begin{figure}[h]
\begin{tikzpicture}[scale=0.9] 
\begin{scope}[every node/.style={circle,draw,thick,minimum size=20pt}]
\node (11) at (0,2) {$o$};
\node (14) at (14.66,2) {$s_3$};

\node (21) at (0,0) {$v$};
\node (22) at (1.33,0) {$t_1$};
\node (23) at (2.66,0) {$t_2$};
\node (24) at (4,0) {$t_3$};
\node (25) at (5.33,0) {$t_4$};
\node (26) at (6.66,0) {$t_5$};
\node (27) at (8,0) {$t_6$};
\node (28) at (9.33,0) {$t_7$};
\node (29) at (10.66,0) {$t_8$};
\node (30) at (12,0) {$t_9$};
\node (31) at (13.33,0) {$t_{10}$};
\node (32) at (14.66,0) {$t_{11}$};
\end{scope}

\begin{scope}[every node/.style={circle}]
	\node (15) at (16,2) {};
	\node (33) at (16,0) {};
\end{scope}

\begin{scope}[>={Stealth[black]},every edge/.style={}, every node/.style={font=\huge}]
\path[->] (14) edge node {$\;\hdots$}(15);
\path[->] (32) edge node {$\;\hdots$}(33);
\end{scope}

\begin{scope}[every node/.style={circle}, every edge/.style={draw,thick}]
\path[->] (11) edge (14);
\path[->] (22) edge node[above]{$\;2$} (21);
\path[->] (23) edge node[above]{$\;\frac{1}{2}$} (22);
\path[->] (24) edge (23);
\path[->] (25) edge node[above]{$\;2$} (24);
\path[->] (26) edge node[above]{$\;\frac{1}{2}$} (25);
\path[->] (27) edge (26);
\path[->] (28) edge (27);
\path[->] (29) edge (28);
\path[->] (30) edge (29);
\path[->] (31) edge (30);
\path[->] (32) edge node[above]{$\;2$} (31);

\path[->] (14) edge node[right]{$\!\!\frac{1}{2}$} (32);
\end{scope}
\end{tikzpicture}
\caption{The second modified copy $G_2$ of $G_0$.}
\label{fig:laddergraphG_2}
\end{figure}
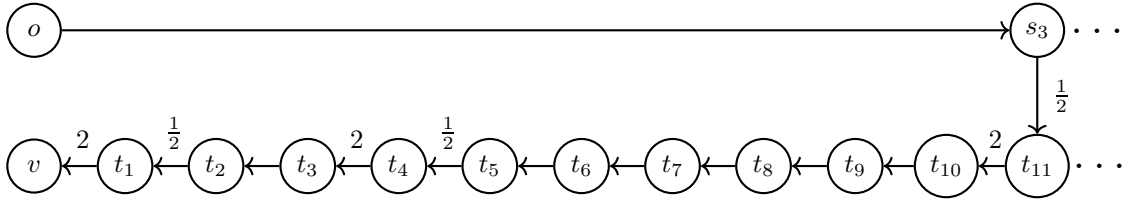

Finally, we connect all the graphs $G_k$ appropriately to obtain a graph/operator with all the desired properties.

\begin{example}
\label{exa:c0}
There exists a positive $T\colon c_0 \to c_0$ with $\norm{T^n}\leq 4$ for every $n\in\N$ such that all powers of $T$ are mean ergodic but $T$ is not weakly almost periodic.

Consider the $c_0$-graph shown in Figure \ref{fig:graph}; it obviously satisfies conditions (a) and (b) above and we denote its associated operator on $c_0$ by $T$. 
There exist at most two paths of equal length with a common endpoint and non-zero weight; the weight of each of them is bounded by $2$. Thus, by Lemma \ref{lem:normTn},
$\norm{T^n}\leq 4$ for all $n\in\N$.

\begin{figure}[h]
\begin{tikzpicture}[scale=0.9] 
\begin{scope}[every node/.style={circle,draw,thick,minimum size=20pt}]
\node (201) at (8,0) {$s$};
\node (202) at (8,2) {$o_0$};
\node (203) at (8,4) {};
\node (204) at (8,8) {};
\node (205) at (8,13.5) {};

\node (211) at (10,2) {$v_0$};
\node (212) at (10,4) {};
\node (213) at (10,5.33) {};
\node (214) at (10,6.66) {};
\node (215) at (10,8) {};
\node (218) at (10,12) {};
\node (219) at (10,13.5) {};
\node (210) at (10,9.5) {};
\end{scope}

\begin{scope}[every node/.style={circle}]
\node (230) at (8,15) {};
\node (231) at (10,15) {};
\node (216) at (10,10.5) {};
\node (217) at (10,11) {};
\end{scope}

\begin{scope}[every node/.style={circle}, every edge/.style={draw,thick}]
\path[->] (201) edge (202);
\path[->] (202) edge (203);
\path[->] (203) edge (204);
\path[->] (204) edge (205);
\path[->] (219) edge node[right] {$2$} (218);
\path[->] (218) edge (217);
\path[->] (216) edge (210);
\path[->] (210) edge node[right]{$\frac{1}{2}$}(215);
\path[->] (215) edge node[right]{$2$}(214);
\path[->] (214) edge (213);
\path[->] (213) edge node[right]{$\frac{1}{2}$}(212);
\path[->] (212) edge node[right]{$2$}(211);
\path[->] (203) edge node[above]{$\frac{1}{2}$}(212);
\path[->] (204) edge node[above]{$\frac{1}{2}$}(215);
\path[->] (205) edge node[above]{$\frac{1}{2}$}(219);
\end{scope}

\begin{scope}[>={Stealth[black]},every edge/.style={}, every node/.style={font=\huge}]
\path[->] (205) edge node {$\vdots$}(230);
\path[->] (219) edge node {$\vdots$}(231);
\path[->] (217) edge node {\small 4 more vertices}(216);
\end{scope}

\begin{scope}[every node/.style={circle,draw,thick,minimum size=20pt}]
\node (102) at (4,2) {$o_1$};
\node (104) at (4,8) {};
\node (105) at (4,13.5) {};
\node (111) at (6,2) {$v_1$};
\node (112) at (6,4) {};
\node (113) at (6,5.33) {};
\node (114) at (6,6.66) {};
\node (115) at (6,8) {};
\node (118) at (6,12) {};
\node (119) at (6,13.5) {};
\node (110) at (6,9.5) {};
\end{scope}

\begin{scope}[every node/.style={circle}]
\node (130) at (4,15) {};
\node (131) at (6,15) {};
\node (116) at (6,10.5) {};
\node (117) at (6,11) {};
\end{scope}

\begin{scope}[every node/.style={circle}, every edge/.style={draw,thick}]
\path[->] (202) edge [bend left] (102);
\path[->] (102) edge (104);
\path[->] (104) edge (105);
\path[->] (119) edge node[right] {$2$} (118);
\path[->] (118) edge (117);
\path[->] (116) edge (110);
\path[->] (110) edge node[right]{$\frac{1}{2}$}(115);
\path[->] (115) edge node[right]{$2$}(114);
\path[->] (114) edge (113);
\path[->] (113) edge node[right]{$\frac{1}{2}$} (112);
\path[->] (112) edge node[right]{$2$} (111);
\path[->] (104) edge node[above]{$\frac{1}{2}$}(115);
\path[->] (105) edge node[above]{$\frac{1}{2}$}(119);
\end{scope}

\begin{scope}[>={Stealth[black]},every edge/.style={}, every node/.style={font=\huge}]
\path[->] (105) edge node {$\vdots$}(130);
\path[->] (119) edge node {$\vdots$}(131);
\path[->] (117) edge node {\small 4 more vertices}(116);
\end{scope}

\begin{scope}[every node/.style={circle,draw,thick,minimum size=20pt}]
\node (2) at (0,2) {$o_2$};
\node (5) at (0,13.5) {};
\node (11) at (2,2) {$v_2$};
\node (12) at (2,4) {};
\node (13) at (2,5.33) {};
\node (14) at (2,6.66) {};
\node (15) at (2,8) {};
\node (18) at (2,12) {};
\node (19) at (2,13.5) {};
\node (10) at (2,9.5) {};
\end{scope}

\begin{scope}[every node/.style={circle,minimum size=0pt}]
\node (32) at (-2,2) {};
\node (30) at (0,15) {};
\node (31) at (2,15) {};
\node (16) at  (2,10.5) {};
\node (17) at  (2,11) {};
\end{scope}

\begin{scope}[every node/.style={circle}, every edge/.style={draw,thick}]
\path[->] (102) edge [bend left] (2);
\path[->] (2) edge (5);
\path[->] (19) edge node[right] {$2$} (18);
\path[->] (18) edge (17);
\path[->] (16) edge (10);
\path[->] (10) edge node[right]{$\frac{1}{2}$} (15);
\path[->] (15) edge node[right]{$2$} (14);
\path[->] (14) edge (13);
\path[->] (13) edge node[right]{$\frac{1}{2}$} (12);
\path[->] (12) edge node[right]{$2$} (11);
\path[->] (5) edge node[above]{$\frac{1}{2}$}(19);
\end{scope}

\begin{scope}[>={Stealth[black]},every edge/.style={}, every node/.style={font=\huge}]
\path[->] (32) edge node {$\hdots$}(2);
\path[->] (5) edge node {$\vdots$}(30);
\path[->] (19) edge node {$\vdots$}(31);
\path[->] (17) edge node {\small 4 more vertices}(16);
\end{scope}

\node at (0.9,6) {$G_2$};
\node at (4.9,6) {$G_1$};
\node at (8.9,6) {$G_0$};
\end{tikzpicture}
\caption{}
\label{fig:graph}
\end{figure}

Now we prove that the operator $T$ is not almost weakly periodic.  To this end, we show that the orbit $\{T^n e_s : n\in\N\}$
is not relatively weakly compact. By the Eberlein-\v{S}mulian theorem, it is sufficient to show that there exists a subsequence of
$(T^ne_s)_{n\in\N}$ which does not posses a weakly convergent subsequence.
Taking into account the previous considerations concerning the subgraphs $G_k$ from Figures \ref{fig:laddergraph} and \ref{fig:laddergraphG_2},
one observes that for every $k\in \N_0$
\[ (T^n e_s)_{v_k} = \begin{cases} 1 & \text{if }n=2^{m+2} \text{ for some }m\geq k\\
0 & \text{otherwise.} \end{cases}
\]
In particular, $(T^{2^m}e_s)_{v_k} \to 1$ as $m \to \infty$ for all $k\in\N_0$. This shows that the sequence
$(T^{2^m}e_s)_{m\in\N}$ has no subsequence that converges pointwise to an element of $c_0$. In particular, no subsequence 
of $(T^{2^m}e_s)$ is weakly convergent.

It remains to show that $T$ and all of its powers are mean ergodic. To this end, let $T^* \colon \ell^1 \to \ell^1$ denote the adjoint of $T$.
It is easy to check that 
\[ \big(T^*(y_n)\big)_u = \biggr(\sum_{n=1}^\infty y_n w(u,n)\biggr)_u \text{ for all }u\in \N,\]
i.e.\ its action is described by the same graph as $T$ but with contrarily directed edges. Let $y\in \ell^1$ such that
$T^* y=y$.  Since $(T^*y)_{v_k} = 0$ for all $k\in \N_0$, it follows that $y$ vanishes on the right-hand side of any $G_k$ in Figure \ref{fig:graph},
i.e.\ on all those vertices previously labeled with $t_1,t_2,t_3,\dots$ in Figure \ref{fig:laddergraph} and \ref{fig:laddergraphG_2}.
Since $y$ is a null sequence, this readily implies $y=0$.
Hence, the fixed space of $T$ separates the fixed space of $T^*$ and the mean ergodic theorem \cite[Thm 8.20]{haase2015} implies that 
$T$ is mean ergodic with mean ergodic projection $0$. It thus follows from Theorem~\ref{thm:ergodic-powers} below that $T^m$ is also mean ergodic for every $m \in \N$.
\end{example}

\begin{rem}
	Let us point out that $\norm{x}_T \coloneqq \sup\{ \norm{T^n \abs{x} } : n\in \N_0\}$ defines an equivalent norm on $c_0$,
	with respect to which the operator $T$ is contractive. Hence, there exists an atomic Banach lattice with order continuous norm $E$
	and a positive contraction $T \colon E \to E$ which is mean ergodic but not weakly almost periodic.
	However, it is not known to the authors if such an operator exists on $c_0$ endowed with the $\infty$-norm.
\end{rem}

\section{A counterexample on $\ell^\infty$}
\label{sec:linfty}

The following example gives a negative answer to Question~6.(iii) in \cite{emelyanov2009}. It is an adaptation of \cite[Exa~1, Sec~V.5]{schaefer1974}.

\begin{example}
\label{exa:linfty}
There exists a positive and contractive operator $T \colon \ell^\infty \to \ell^\infty$ such that $T$ is ergodic while $T^{2j}$ is not for any $j\in\N$.
In particular, $T$ is not weakly almost periodic.
\end{example}
\begin{proof}
	We first consider the space $F\coloneqq (\R^2,\norm{\argument}_\infty)$. The matrix 
	\[ S\coloneqq \frac{1}{\sqrt{2}} \begin{pmatrix} 1 & 1 \\ 1 & -1 \end{pmatrix} \]
	defines	a bijective linear transformation on $F$ such that $S^{-1}=S^T=S$.
	Now let $a_m \coloneqq \big(1-\frac{1}{m}\big)$ for $m\in\N$ and define the operators
	\[ T_m \coloneqq S \begin{pmatrix} 1 & 0 \\ 0 & -a_m \end{pmatrix} S^{-1} = \frac{1}{2} \begin{pmatrix} \frac{1}{m} & 2 - \frac{1}{m} \\ 2-\frac{1}{m} & \frac{1}{m} \end{pmatrix} \]
	on $F$. Clearly, every $T_m$ is positive and contractive on $F$ and 
	\[ \frac{1}{n} \sum_{k=0}^{n-1} T_m^k = S \cdot \frac{1}{n} \sum_{k=0}^{n-1} \begin{pmatrix} 1 & 0 \\ 0 & (-a_m)^k \end{pmatrix} \cdot S^{-1}.\]
	for each $n \in \N$. As $\abs{\sum_{k=0}^n (-a_m)^k}\leq 1$ for all $n\in\N$ and every $m\in\N$, the Ces\`aro averages of $T_m$ converge to 
	\[U \coloneqq S \begin{pmatrix} 1 & 0 \\ 0 & 0 \end{pmatrix} S^{-1} = \frac{1}{2} \begin{pmatrix} 1 & 1 \\ 1 & 1 \end{pmatrix}\]
	and even uniformly in $m\in\N$.
	Now we define the operators $T\coloneqq T_1 \oplus T_2 \oplus T_3 \oplus\dots$ and $\cU \coloneqq U \oplus U \oplus U \oplus \dots$ 
	on the space $F\oplus F \oplus F \oplus \dots$ of all bounded sequences in $F$ which is, when endowed with the supremum norm, isometrically lattice isomorphic to $\ell^\infty$. It follows
	from the previous considerations that the Ces{\`a}ro means $\frac{1}{n}\sum_{k=0}^{n-1} T^k$ converge to $\cU$ with respect to the operator norm as $n \to \infty$. 

	Now we turn our attention to even powers of the operator $T$. Fix $j\in \N$. The operator $T^{2j}$ is similar to the operator
	\[ \begin{pmatrix} 1 & 0 \\ 0 & (-a_1)^{2j} \end{pmatrix} \oplus \begin{pmatrix} 1 & 0 \\ 0 & (-a_2)^{2j} \end{pmatrix} \oplus \begin{pmatrix} 1 & 0 \\ 0 & (-a_3)^{2j} \end{pmatrix} \oplus \dots\]
	on $\ell^\infty$. Hence, if $T^{2j}$ was mean ergodic, then so was the multiplication operator $M \colon \ell^\infty \to \ell^\infty$ given by $M(x_m) = \big((a_m)^{2j}x_m\big)$ 
	for all $(x_m) \in \ell^\infty$. 
	Since $M$ is positive and has spectral radius $1$, its adjoint $M^*$ has a non-zero fixed point \cite[p.\ 705]{schaefer1967} but obviously the only fixed point of $M$ is $0$. 
	Thus, $M$ cannot be mean ergodic and neither can $T^{2j}$. 
	Due to the classical mean ergodic theorem, this implies that $T$ is
	not weakly almost periodic.
\end{proof}

\begin{rem}
	An alternative way to see that $T^{2j}$ is not mean ergodic for any $j \in \N$ in Example~\ref{exa:linfty} is by direct computation: the $n$-th Ces\'{a}ro mean $\frac{1}{n} \sum_{k=0}^{n-1} (T^{2j})^k$ of $T^{2j}$, applied to the vector 
	$(
	S
	\begin{pmatrix}
		0 \\ 1
	\end{pmatrix}, \;
	S
	\begin{pmatrix}
		0 \\ 1
	\end{pmatrix},
	\dots
	)$,
	is given by
	$(
	S
	\begin{pmatrix}
		0 \\ b_{1,n}
	\end{pmatrix}, \;
	S
	\begin{pmatrix}
		0 \\ b_{2,n}
	\end{pmatrix},
	\dots
	)$,
	where $b_{m,n} \coloneqq \frac{1}{n} \frac{1 - a_m^{2jn}}{1 - a_m^{2j}}$ for all $m,n \in \N$. For each fixed $m$, the number $b_{m,n}$ converges to $0$ as $n \to \infty$, but it is easy to see that the convergence is not uniform with respect to $m$ (since $a_m^{2j}$ is close to $1$ for large $m$). Hence, $T^{2j}$ is not mean ergodic.
\end{rem}

As a consequence of Examples \ref{exa:c0} and \ref{exa:linfty} we obtain a sufficient condition for an order complete Banach lattice to be a KB-space; see Theorem~\ref{thm:KB} below. Recall that a Banach lattice $E$ is called a \emph{KB-space} if every norm bounded increasing sequence (equivalently: net) in the positive cone $E_+$ is norm convergent. Important examples of KB-spaces are the class of all reflexive Banach lattices and the class of all $L^1$-spaces, but there are also more complicated examples (see for instance \cite[p.\,95]{meyer1991}). For a thorough treatment of KB-spaces we refer the reader to \cite[Section~2.4]{meyer1991}; here we only recall that every KB-space has order continuous norm and that a Banach lattice $E$ is a KB-space if and only if $E$ does not contain any sublattice isomorphic to $c_0$ \cite[Theorem~2.4.12]{meyer1991}.

\begin{thm}
\label{thm:KB}
	Let $E$ be a countably order complete Banach lattice such that every positive, power-bounded and mean ergodic operator $T$ on $E$ is weakly almost periodic. Then $E$ is
	a KB-space.
\end{thm}

It is known that a countably order complete Banach lattice that does not have order continuous norm contains a sublattice which is isomorphic to $\ell^\infty$ \cite[Cor~2.4.3]{meyer1991}. 
For the proof of Theorem~\ref{thm:KB} we need a bit more, namely that there exists a sublattice which is isomorphic to $\ell^\infty$ and which is at the same time the range of a positive projection. 
Since we could not find an explicit reference for this observation in the literature, we include a proof in the following lemma. 
The Banach lattice isomorphism in the statement of the lemma is not necessarily isometric.

\begin{lem} \label{lem:linfty-as-sublattice}
	Let $E$ be a countably order complete Banach lattice which does not have order continuous norm. Then there exists a closed vector sublattice $F$ of $E$ with the following properties:
	\begin{enumerate}[(a)]
		\item $F$ is the range of a positive projection $P \colon E \to E$.
		\item There exists a Banach lattice isomorphism $i \colon \ell^\infty \to F$.
	\end{enumerate}
\end{lem}
\begin{proof}
	By \cite[Thm 2.4.2]{meyer1991} there exists an order bounded disjoint sequence $(x_n)\subseteq E_+$ which does not converge to zero. Without loss of generality
	we may assume that $\norm{x_n}=1$ for each $n\in\N$.
	Now pick a sequence of positive functionals $\phi_n\in E'_+$ such that $\norm{\phi_n}=1$ and $\applied{\phi_n}{x_n}=1$ for every $n\in\N$.
	By the countably order completeness, every principle band is a projection band. 
	If $P_n\colon E \to \{x_n\}^{\bot\bot}$ denotes the band projection onto $\{x_n\}^{\bot\bot}$, then
	each $\psi_n\coloneqq \phi_n \circ P_n$ is a positive functional on $E$ such that
	$\norm{\psi_n}=1$. Define
	\[ i(a_n)  \coloneqq \sup\{ a_n x_n : n\in \N\}\]
	for every sequence $(a_n) \in \ell^\infty_+$. According to (the proof of) \cite[Lem 2.3.10(ii)]{meyer1991}, 
	$i$ extends to a lattice homomorphism $i\colon \ell^\infty \to E$ with closed range $F \coloneqq i(\ell^\infty)$ and $i \colon \ell^\infty \to F$ is a Banach lattice isomorphism.
	
	Clearly, $\tau \colon E \to \ell^\infty$, given by $\tau(x) \coloneqq (\psi_n(x))_{n\in\N}$, is a positive operator and it follows right 
	from the definitions that $\tau \circ i  = \id_{\ell^\infty}$. Hence, $P \coloneqq i \circ \tau$ is a positive projection on $E$ with range $F$.
\end{proof}

\begin{proof}[Proof of Theorem~\ref{thm:KB}]
	Let us first consider the case that $E$ does not have order continuous norm. Then, according to Lemma~\ref{lem:linfty-as-sublattice}, there exists a closed sublattice $F \subseteq E$, 
	a Banach lattice isomorphism $i\colon \ell^\infty \to F$ and a positive projection $P \colon E \to E$ with range $F$. 
	Now let $T\colon \ell^\infty \to \ell^\infty$ denote the operator from Example \ref{exa:linfty} and define $S\colon E\to E$ by
	$S \coloneqq iTi^{-1}P$. 
	It is easy to verify that $S$ is a positive, power-bounded and mean ergodic operator which is not weakly almost periodic. 

	Now let $E$ be Banach lattice with order continuous norm which is no KB-space. Then, by \cite[Thm 2.4.12]{meyer1991}, $E$ contains a copy of $c_0$
	meaning that there exists a closed sublattice $F\subseteq E$ and a Banach lattice isomorphism $i\colon c_0 \to F$.
	Moreover, by \cite[Cor 2.4.3]{meyer1991}, there exists a positive projection $P\colon E \to E$ with range $F$.
	Let $T\colon c_0 \to c_0$ denote the operator from Example \ref{exa:c0} and, as above, define $S\colon E\to E$ by $S \coloneqq iTi^{-1}P$.
	Again, it is easy to verify that $S$ is a positive, power-bounded and mean ergodic operator which is not weakly almost periodic.
\end{proof}

\begin{rem}
	The question whether the converse of Theorem \ref{thm:KB} also holds is still open; see also Open Problem 3.1.19 in \cite{emelyanov2007}.
\end{rem}

\section{A mean ergodic theorem} \label{sec:ergodic-powers}

In general the powers of a positive mean-ergodic operator need not be mean ergodic,
even if the operator is power-bounded (see \cite{sine1976}, \cite{emelyanov2009} and Example~\ref{exa:linfty}). 
If the underlying Banach lattice has, however, order continuous norm and $T$ is positive and mean ergodic, then each power of $T$ is mean ergodic, too (see \cite[Prop 4.5]{derriennic1981} 
for the case of power-bounded operators and see \cite[Thm 12]{emelyanov2003} for the general case, even on more general ordered Banach spaces). 
We close this article by proving another mean ergodic theorem of this type. Instead of imposing a condition on the Banach lattice we assume that the mean ergodic projection of the operator is $0$.

\begin{thm} \label{thm:ergodic-powers}
	Let $T$ be a positive operator on a Banach lattice $E$ and assume that $T$ is mean ergodic with mean ergodic projection $0$. Then every power of $T$ is also mean ergodic with mean ergodic projection $0$.
\end{thm}
\begin{proof}
	Fix $m \in \N$. By the uniform boundedness principle the Ces{\`a}ro means of $T$ are uniformly bounded, and $(T^n/n)$ converges strongly to $0$ as $n \to \infty$. 
	Hence, $(T^{mn}/n)$ also converges strongly to $0$ as $n \to \infty$ and it follows from the positivity of $T$ that the Ces{\`a}ro means of $T^m$ are uniformly bounded, too. 
	Therefore, we only have to show that the fixed space of $(T^m)^*$ equals $\{0\}$, i.e.\ that $1$ is not an eigenvalue of $(T^m)^*$.
	To this end, it suffices to prove that $T^*$ has no eigenvalues on the unit circle.
	
	So assume for a contradiction that $\lambda$ is an eigenvalue of $T^*$ with $\abs{\lambda} = 1$ and let $\varphi$ be an associated eigenvector. 
	Since $\abs{ \varphi } = \abs{ T^*\varphi } \leq T^* \abs{ \varphi }$, the sequence $\big((T^*)^n\abs{ \varphi}\big)_{n \in \N_0}$ is increasing. 
	In particular, the Ces{\`a}ro means of $T^*$ applied to $\abs{\varphi}$ are all larger than $\abs{ \varphi }$. 
	On the other hand, these Ces{\`a}ro means are weak$^*$ convergent to $0$ since $T$ is mean ergodic with mean ergodic projection $0$. This is a contradiction as $\abs{ \varphi } \not= 0$.
\end{proof}

\begin{rem}
	The proof of Theorem~\ref{thm:ergodic-powers} also shows that $\lambda T$ is mean ergodic with mean ergodic projection $0$ for every complex number $\lambda$ of modulus $1$.
\end{rem}

\bibliographystyle{abbrv}
\bibliography{literature}

\begin{thebibliography}{10}

\bibitem{abramovich2002}
Y.~A. Abramovich and C.~D. Aliprantis.
\newblock {\em An invitation to operator theory}, volume~50 of {\em Graduate
  Studies in Mathematics}.
\newblock American Mathematical Society, Providence, RI, 2002.

\bibitem{alpay2006}
S.~Alpay, A.~Binhadjah, E.~Y. Emelyanov, and Z.~Ercan.
\newblock Mean ergodicity of positive operators in {KB}-spaces.
\newblock {\em J. Math. Anal. Appl.}, 323(1):371--378, 2006.

\bibitem{bermudez2000}
T.~Berm\'udez, M.~Gonz\'alez, and M.~Mbekhta.
\newblock Operators with an ergodic power.
\newblock {\em Studia Math.}, 141(3):201--208, 2000.

\bibitem{derriennic1981}
Y.~Derriennic and U.~Krengel.
\newblock Subadditive mean ergodic theorems.
\newblock {\em Ergodic Theory Dynamical Systems}, 1(1):33--48, 1981.

\bibitem{haase2015}
T.~Eisner, B.~Farkas, M.~Haase, and R.~Nagel.
\newblock {\em {O}perator {T}heoretic {A}spects of {E}rgodic {T}heory}, volume
  272 of {\em Graduate Texts in Mathematics}.
\newblock Springer-Verlag, New York, 2015.

\bibitem{emelyanov2007}
E.~Y. Emel'yanov.
\newblock {\em Non-spectral asymptotic analysis of one-parameter operator
  semigroups}, volume 173 of {\em Operator Theory: Advances and Applications}.
\newblock Birkh\"auser Verlag, Basel, 2007.

\bibitem{emelyanov2009}
E.~Y. Emel'yanov and N.~Erkursun.
\newblock An extension of {S}ine's counterexample.
\newblock {\em Positivity}, 13(1):125--127, 2009.

\bibitem{emelyanov2003}
E.~Y. Emel'yanov and M.~P.~H. Wolff.
\newblock Positive operators on {B}anach spaces ordered by strongly normal
  cones.
\newblock {\em Positivity}, 7(1-2):3--22, 2003.
\newblock Positivity and its applications (Nijmegen, 2001).

\bibitem{foguel1964}
S.~R. Foguel.
\newblock A counterexample to a problem of {S}z.-{N}agy.
\newblock {\em Proc. Amer. Math. Soc.}, 15:788--790, 1964.

\bibitem{komornik1993}
J.~Komorn{\'\i}k.
\newblock Asymptotic periodicity of {M}arkov and related operators.
\newblock In {\em Dynamics reported}, volume~2 of {\em Dynam. Report.
  Expositions Dynam. Systems (N.S.)}, pages 31--68. Springer, Berlin, 1993.

\bibitem{kornfeld2000}
I.~Kornfeld and M.~Lin.
\newblock Weak almost periodicity of {$L_1$} contractions and coboundaries of
  non-singular transformations.
\newblock {\em Studia Math.}, 138(3):225--240, 2000.

\bibitem{meyer1991}
P.~Meyer-Nieberg.
\newblock {\em Banach lattices}.
\newblock Universitext. Springer-Verlag, Berlin, 1991.

\bibitem{mueller2007}
V.~M\"uller and Y.~Tomilov.
\newblock Quasisimilarity of power bounded operators and {B}lum-{H}anson
  property.
\newblock {\em J. Funct. Anal.}, 246(2):385--399, 2007.

\bibitem{schaefer1967}
H.~H. Schaefer.
\newblock Invariant ideals of positive operators in {$C(X)$}. {I}.
\newblock {\em Illinois J. Math.}, 11:703--715, 1967.

\bibitem{schaefer1974}
H.~H. Schaefer.
\newblock {\em Banach lattices and positive operators}.
\newblock Springer-Verlag, New York, 1974.
\newblock Die Grundlehren der mathematischen Wissenschaften, Band 215.

\bibitem{sine1976}
R.~Sine.
\newblock A note on the ergodic properties of homeomorphisms.
\newblock {\em Proc. Amer. Math. Soc.}, 57(1):169--172, 1976.

\bibitem{oosterhout2009}
M.~van Oosterhout.
\newblock The blum-hanson property.
\newblock Master's thesis, Delft University of Technology, 2009.

\end{thebibliography}

\end{document}